\documentclass{amsart}
\usepackage{amssymb}
\newcommand{\T}{{\mathcal S}}

\newcommand{\Prob}{{\operatorname{\mathsf {Prob}}}}

\newcommand{\rest}{{\mathord{\restriction}}}

\newcommand{\dom}{{\operatorname{\mathsf {dom}}}}

\newcommand{\N}{{\mathcal N}}

\newcommand{\<}{\langle}
\renewcommand{\>}{\rangle}

\newcommand{\lft}[2]{\mathopen\ifcase#1{}\oo\or
                        \big#2\or\Big#2\else\oo\fi}
\newcommand{\rgt}[2]{\mathclose\ifcase#1{}\oo\or
                        \big#2\or\Big#2\else\oo\fi}

\theoremstyle{plain}
\newtheorem{theorem}{Theorem}
\theoremstyle{plain}
\newtheorem{lemma}[theorem]{Lemma}

\newtheorem{definition}[theorem]{Definition}

\begin{document}
\mbox{}
\title{A note on small sets of reals}
\author{Tomek Bartoszynski}
\address{National Science Foundation\\
Division of Mathematical Sciences\\
2415 Eisenhower Avenue\\
Alexandria, VA 22314\\
USA}

\email{tbartosz@nsf.gov}
\author{Saharon Shelah}
\thanks{ Second author was partially supported by
  NSF grants DMS-1101597 and DMS-136974, publication 1139}
\address{Department of Mathematics\\
Hebrew University\\
Jerusalem, Israel}
\email{shelah@math.huji.ac.il, http://math.rutgers.edu/\char 126 shelah/}
%\subjclass{03,05}
\begin{abstract}
We construct an example of a combinatorially large measure zero set.
\end{abstract}
\maketitle

\section{Introduction.}
 We will work in the space $2^\omega$ equipped with standard topology
 and measure. More specifically, the topology is generated by basic
 open sets of form $[s]=\{x \in 2^\omega: s \subset x\}$ for $s \in
 2^a, \ a\in \omega^{<\omega}$. The measure is the standard product measure such
 that $\mu([s])=2^{-|\dom(s)|}$ and let $\N$ be the
 collection of all measure zero sets.

Measure zero sets in $2^\omega$ admit the following representation
(see lemma \ref{basic}):

$X \in \N$  iff and only
if there exists a sequence $\{F_n: n \in \omega\}$ such that
\begin{enumerate}
\item $F_n \subseteq 2^n$ for $n \in \omega$,
\item $\sum_{n \in \omega} \dfrac{|F_n|}{2^n} < \infty$,
\item $X \subseteq \{x \in 2^\omega: \exists^\infty n \ x \rest n \in
  F_n\}$.
\end{enumerate}

The main drawback of this representation is that sets $F_n$ have
overlapping domains. The following definitions from \cite{Bar88Cov} and \cite{BJbook}
offer a refinement.

\begin{definition}
\begin{enumerate}
\item A set $X \subseteq 2^\omega$ is {\em small} ($X \in \T$) if there exists a sequence
$\{I_n,J_n: n\in \omega\}$ such that
\begin{enumerate}
\item $I_n \in [\omega]^{<\aleph_0} $ for $n \in \omega$,
\item $I_n \cap I_m=\emptyset$ for $n \neq m$,
\item $J_n \subseteq 2^{I_n}$ for $n \in \omega$,
\item $\sum_{n \in \omega} \dfrac{|J_n|}{2^{|I_n|}} < \infty$
\item $X \subseteq \{x \in 2^\omega: \exists^\infty n \ x \rest I_n \in
  J_n\}$
\end{enumerate}
\text{Without loss of generality we can assume that $\{I_n: n \in \omega\}$
is a partition of $\omega$ into finite sets.}
\item We say that $X$ is {\em small}$^\star$ ($X \in \T^\star$) if in
addition sets
$I_n$ are disjoint  intervals, 
that is,  if there exists a strictly increasing
sequence of integers $\{k_n: n \in \omega\}$ such that
$I_n=[k_n,k_{n+1})$ for each $n$.
\end{enumerate}
Let $(I_n,J_n)_{n \in \omega}$ denote the set $\{x \in 2^\omega:
\exists^\infty n \ x \rest n \in 
  J_n\}$.
\end{definition}

It is clear that $\T^\star \subseteq \T \subseteq \N$.

Small sets are useful because of their combinatorial simplicity. To
test that $x \in X \in \T$ the real $x$ must pass infinitely many {\it
  independent} tests as in Borel-Cantelli lemma. Furthermore, various
structurally simple measure zero
sets are small.
In
particular,
\begin{enumerate}
\item if $X \in \N$ and $|X|<2^{\aleph_0} $ then $X \in \T^\star$,
  \cite{Bar88Cov}
\item if $X$ is contained in a countable union of closed measure zero
  sets then $X \in \T^\star$, \cite{BJbook}
\item if $F $ is a filter on $\omega$ (interpreted as a subset of
  $2^\omega$) and $F \in \N$ then $F \in \T^\star$, \cite{Bar92Stru}, \cite{Tal2013}
\end{enumerate}

\begin{definition}
For  families of sets ${\mathcal A}, {\mathcal B}$ let ${\mathcal A} \oplus {\mathcal
  B}$ be 
$$\{ X : \exists a \in {\mathcal A} \ \exists b  \in {\mathcal B} \ (X
\subset a\cup b) \}$$
\end{definition} 

Clearly, if ${\mathcal J}$ is an ideal then ${\mathcal J} \oplus
{\mathcal J} = {\mathcal J}$. Likewise, ${\mathcal A} \cup ({\mathcal
  A}\oplus {\mathcal A}) \cup ({\mathcal
  A}\oplus {\mathcal A}\oplus {\mathcal A}) \cup \dots$ is an ideal
for any ${\mathcal A}$.

\begin{theorem}\label{old}\cite{Bar88Cov}

$\T^\star\oplus \T^\star= \T \oplus \T= \N = \N\oplus \N$.

\end{theorem}

The main result of this paper is to show that the above result is best
possible, that is 
$\T^\star \subsetneq \T \subsetneq \N$. It was known (\cite{Bar88Cov})
that $\T^\star \subsetneq \N$.

\section{Preliminaries}

To make the paper complete and self contained we present a review of
known results.

\begin{lemma}\label{basic}
Suppose that $X \subset 2^\omega$. $X$ has measure zero iff and only
if there exists a sequence $\{F_n: n \in \omega\}$ such that
\begin{enumerate}
\item $F_n \subseteq 2^n$ for $n \in \omega$,
\item $\sum_{n \in \omega} \dfrac{|F_n|}{2^n} < \infty$,
\item $X \subseteq \{x \in 2^\omega: \exists^\infty n \ x \rest n \in
  F_n\}$.
\end{enumerate}
\end{lemma}
\begin{proof}
$\longleftarrow$ Note that 
$\{x \in 2^\omega: \exists^\infty n \ x \rest n \in
  F_n\} = \bigcap_{m \in \omega} \bigcup_{n \geq m} \{x \in 2^\omega:  x \rest n\in
  F_n\}$. 
Now,
$$\mu\left(\bigcup_{n \geq m} \{x \in 2^\omega:  x \rest n\in
  F_n\}\right)\leq \sum_{n \geq m}  \mu\left(\{x \in 2^\omega:  x \rest n\in
  F_n\}\right) \leq \sum_{n\geq m} \dfrac{|F_n|}{2^n}
\longrightarrow
0.$$

$\longrightarrow$ If $X$ has measure zero then there exists a sequence
of open sets $\{U_n: n \in \omega\}$ such that 
\begin{enumerate}
\item $\mu(U_n) \leq 2^{-n}$, for each $n$,
\item $X \subseteq \bigcap_{n\in \omega} U_n$.
\end{enumerate}
Find a sequence of $\{s^n_m: n,m\in \omega\}$ such that
\begin{enumerate}
\item $s^n_m \in 2^{<\omega}$,
\item $[s^n_m] \cap [s^n_k] =\emptyset$ when $k \neq m$,
\item $U_n =\bigcup_{m \in \omega} [s^n_m]$.
\end{enumerate}

For $k \in \omega$ let $F_k=\{s^n_m: n,m \in \omega, \
|s^n_m|=k\}$. Note that $X \subseteq \{x \in 2^\omega: \exists^\infty
k \ x \rest k \in F_k\}$ and that $\sum_{k \in \omega}
\dfrac{|F_k|}{2^k} \leq \sum_{n \in \omega} \mu(U_n) \leq 1$.
\end{proof}

\begin{theorem}\label{old1}\cite{Bar88Cov}
$\T^\star\oplus \T^\star= \T \oplus \T= \N$.
\end{theorem}
\begin{proof}
Since $\N$ is an ideal, $\N \oplus \N=\N$. Consequently, it suffices
to show that $\T^\star\oplus \T^\star= \N$. The 
following theorem gives the required decomposition.

\begin{theorem}[\cite{Bar88Cov}]
Suppose that $X \subseteq 2^\omega$ is a measure zero set. Then there
exist sequences $\<n_k, m_k : k \in \omega\>$ and $\<J_k, J'_k :
k \in \omega\>$ such that 
\begin{enumerate}
\item $\displaystyle n_k < m_k < n_{k+1} $ for all $k \in \omega$,
\item $\displaystyle J_k \subseteq 2^{[n_k, n_{k+1})}$, $J'_k
  \subseteq 2^{[m_k, 
m_{k+1})}$ for $k \in \omega$,
\item the sets $\displaystyle \lft1([n_k,n_{k+1}),J_k\rgt1)_{k \in \omega} $ and
$\displaystyle \lft1([m_k,m_{k+1}),J'_k\rgt1)_{k \in \omega} $ are small$^\star$, and
\item $\displaystyle X \subseteq
  \lft1([n_k,n_{k+1}),J_k\rgt1)_{k \in \omega}  \cup 
\lft1([m_k,m_{k+1}),J'_k\rgt1)_{k \in \omega} $.
\end{enumerate}
In particular, every null set is a union of two small$^\star$
sets.
\end{theorem}
\begin{proof}
Let $X \subseteq 2^{\omega}$ be a null set. 

We can assume
that $X \subseteq \left\{x \in 2^{\omega}: \ \exists^{\infty}n \ x \rest n \in
F_{n}\right\}$ for some sequence $\<F_{n} : n \in \omega\>$ satisfying
conditions of Lemma \ref{basic}.

Fix a sequence of positive reals $\< \varepsilon_{n} : n \in \omega\>$
such that $\sum_{n=1}^{\infty} \varepsilon_{n} < \infty$.

Define two sequences $\<n_{k},m_{k}:k \in \omega\>$ as follows:
$n_{0}=0$,
$$m_{k} = \min\left\{j > n_k :
2^{n_{k}} \cdot \sum_{i=j}^{\infty}
\frac{|F_{i}|}{2^{i}} < \varepsilon_{k}\right\}   ,$$ and
$$n_{k+1} = \min\left\{j > m_k :
2^{m_{k}} \cdot \sum_{i=j}^{\infty}
\frac{|F_{i}|}{2^{i}} < \varepsilon_{k}\right\} \ \hbox{for } k \in
\omega  .$$ 
 
Let
$I_{k}=[n_{k},n_{k+1})$    and  $I'_{k}=[m_{k},m_{k+1})$  for
$k \in \omega$.
Define
\begin{multline*}
s \in  J_{k} \iff   s \in 2^{I_{k}} \ \& \
\exists i \in [m_{k},n_{k+1})\  \exists t \in F_{i} \
s \rest \dom(t) \cap \dom(s)= \\
       t \rest \dom(t) \cap \dom(s)   .
\end{multline*}
Similarly
\begin{multline*}
s \in  J'_{k}\iff   s \in 2^{I'_{k}} \ \& \  \exists i \in
[n_{k+1},m_{k+1})\
\exists t \in F_{i} \  s \rest \dom(t) \cap \dom(s)= \\
       t \rest \dom(t) \cap \dom(s)  . 
\end{multline*}

It remains to show that $(I_{k},J_{k})_{k\in \omega}$ and
$(I'_{k},J'_{k})_{k\in \omega}$ are small sets and
that their union  covers $X$.

Consider the set $(I_{k},J_{k})_{k\in \omega}$. Notice that for $k
\in \omega$ 
$$\frac{|J_{k}|}{2^{I_{k}}} \leq
2^{n_{k}} \cdot \sum^{n_{k+1}}_{i=m_{k}} \frac{|F_{i}|}{2^{i}}
\leq  \varepsilon_{k}.$$
Since $\sum_{n=1}^{\infty} \varepsilon_{n} < \infty$ this shows
that the set $(I_{n},J_{n})_{n\in \omega}$ is null.
An analogous argument shows that $(I'_{k},J'_{k})_{k\in \omega}$ is null.
Finally, we show that
$$X \subseteq (I_{n},J_{n})_{n\in \omega}\cup (I'_{n},J'_{n})_{n\in \omega}
.$$ 
Suppose that $x \in X$ and let
$Z= \left\{n \in \omega: x \rest n \in F_{n}\right\}$. By the choice
of $F_n$'s set $Z$ is infinite. 
Therefore, one of the sets,
$$Z \cap \bigcup_{k\in \omega}[m_{k},n_{k+1}) \quad \text{ or }
Z \cap \bigcup_{k\in \omega} [n_{k+1},m_{k+1}),$$
is infinite.
Without loss of generality we can assume that it is the first set. 
It follows that $x \in  (I_{n},J_{n})_{n\in \omega}$ because if $x \rest n
\in  F_{n}$ and $n \in [m_{k},n_{k+1})$, then by the
definition there is $t \in J_{k}$ such
that $x \rest [n_{k},n_{k+1})=t$.
\end{proof}
\end{proof}

Now lets turn attention to the family of small sets $\T$. Observe that
the representation used in the definition of small sets is not unique.
In particular, it is easy
to see that

\begin{lemma}\label{trivial}
Suppose that
$(I_n,J_n)_{n\in \omega}$ is a small set and $\{a_k: k\in
\omega\}$ is a partition of $\omega$ into finite sets. For $n \in
\omega$ define
$I'_n=\bigcup_{l\in a_n} I_l$ and $J'_n=\{s \in 2^{I'_n}: \exists l
\in a_n \ \exists t \in J_l \ s \rest I_l = t\rest I_l\}$.
Then $(I_n,J_n)_{n \in \omega}=(I'_n,J'_n)_{n \in
  \omega}$.
\end{lemma}

\begin{lemma}\label{union}
Suppose that $(I_n,J_n)_{n\in \omega}$ and
$(I'_n,J'_n)_{n\in \omega}$ are two small sets.
If $\{I_n: n\in \omega\}$ is a finer partition than $\{I'_n: n\in \omega\}$, then
$(I_n,J_n)_{n\in \omega} \cup
(I'_n,J'_n)_{n\in \omega})$ is a small set.
\end{lemma}
\begin{proof}
Define $I''_n=I'_n$  for $n \in \omega$ and let   
$$J''_{n} = J'_{n} \cup \left\{s \in 2^{I'_n} : \exists k \ \exists s
  \in J_k  (
I_k \subseteq I'_n  \ \& \ s \rest I_k \in J_{k})\right\} .$$
It is easy to see that
$(I_n,J_n)_{n\in \omega} \cup
(I_n,J_n)_{n\in \omega})=(I''_n,J''_n)_{n\in \omega} $.
\end{proof}

 Since members of $\T$ do not seem to form an ideal we are interested
 in characterizing instances when a union of two sets in $\T$ is in $\T$.

\begin{theorem}
Suppose that $(I_n,J_n)_{n \in \omega}$ and $(I'_n,J'_n)_{n \in
  \omega}$ are two small sets and $(I_n,J_n)_{n \in \omega}
\subseteq (I'_n,J'_n)_{n \in  \omega}$. Then there exists a set $(I''_n,J''_n)_{n \in
  \omega}$ such that $(I_n,J_n)_{n \in \omega}
\subseteq (I''_n,J''_n)_{n \in  \omega} \subseteq (I'_n,J'_n)_{n \in
  \omega}$ and partition $\{I''_n: n \in \omega\}$ is finer than both
$\{I_n: n \in \omega\}$ and $\{I'_n: n \in \omega\}$. 
\end{theorem}
\begin{proof}
Let start with the following:
\begin{lemma}\label{key}
Suppose that $(I_n,J_n)_{n \in \omega}$ and $(I'_n,J'_n)_{n \in
  \omega}$ are two small sets.
The following conditions are equivalent:
\begin{enumerate}
\item $(I_n,J_n)_{n \in \omega} \subseteq (I'_n,J'_n)_{n \in
  \omega}$,
\item for all but finitely many $n\in \omega$ and for every $s \in J_n$ there
  exists $m \in \omega$ and $t \in J'_m$ such that
\begin{enumerate}
\item $I_n \cap I'_m \neq \emptyset$,
\item $s \rest (I_n \cap I'_m ) = t \rest (I_n \cap I'_m)$,
\item $\forall u \in 2^{I'_m \setminus I_n} \ t \rest (I_n \cap I'_m) ^\frown u \in J'_m$.
\end{enumerate}
\end{enumerate}
\end{lemma}
\begin{proof}
$(2) \rightarrow (1)$ Suppose that $x \in (I_n,J_n)_{n \in \omega}$. Then
for infinitely many $n$, $x \rest I_n \in J_n$.  For all but finitely
many of those $n's$, conditions (b) and (c) of clause (2) guarantee that for some
$m$ such that $I_n \cap I'_m\neq \emptyset$, $x \rest
(I_n \cap I'_m )^\frown x \rest (I'_m \setminus I_n) \in
J'_m$. Consequently, $x \in (I'_n,J'_n)_{n \in
  \omega}$.

$\neg (2) \rightarrow \neg (1)$ Suppose that condition (2) fails. Then there exists an
infinite set $Z \subseteq \omega$ such that for each $n \in Z$ there is
$s_n \in  J_n$ such that for every $m$ such that $I_n \cap I'_m
\neq \emptyset$ exactly one of the following conditions holds:
\begin{enumerate}
\item $s_n \rest (I_n \cap I'_m ) \neq t \rest (I_n \cap I'_m)$ for every
  $t \in J'_m$,
\item there is $t \in J'_m$ such that $s_n \rest (I_n \cap I'_m ) = t
  \rest (I_n \cap I'_m)$ but for some $u=u_{n,m} \in 2^{I'_m\setminus
    I_n}$, $t \rest (I_n \cap I'_m )^\frown u_{n,m} \not\in J'_m$.
\end{enumerate}
By thinning out the set $Z$ we can assume that no set $I'_m$
intersects two distinct sets $I_n$ for $n \in Z$. Also for each $m \in
\omega$ fix $t^m \in 2^{I'_m}$ such that $t^m \not\in J'_m$.

Let $x \in 2^\omega$ be defined as follows:

$$x(l)=\left\{\begin{array}{ll}
s_n(l) & n \in Z \text{ and } l \in I_n \text{ and } u_{n,m} \text{ is
         not defined}\\
0 & \text{if $n \in Z$ and $l \in I'_m \setminus I_n$ and $I_n\cap I_m\neq
    \emptyset$ and $u_{n,m}$ is not defined}\\
s_n(l) & \text{if $n \in Z$ and $l \in I_n\cap I'_m$ and $u_{n,m}$ is defined}\\
u_{n,m}(l) & \text{if $n \in Z$ and $l \in I'_m\setminus I_n$ and and $I_n\cap I_m\neq
    \emptyset$ and $u_{n,m}$ is
             defined} \\
t^m(l) & \text{if } l\in I_m \text{ and } I_m \cap I_n= \emptyset 
         \text{ for all } n \in Z\end{array}\right. .
$$

Observe that the first two clauses define $x \rest I'_m$ when $I'_m
\cap I_n\neq \emptyset$ for some $n\in Z$ and $u_{n,m}$ is undefined,
the next two clauses define $x \rest I'_m$ when $I'_m
\cap I_n\neq \emptyset$ for some $n\in Z$ and $u_{n,m}$ is defined,
and finally the last clause defines $x \rest I'_m$ when $I'_m \cap
I_n=\emptyset$ for all $n \in Z$.
It is easy to see that these cases are mutually exclusive and that 
$x \in (I_n,J_n)_{n \in \omega}$ since $x \rest I_n =s_n \in J_n$ for $n \in
Z$. Finally note that $x \not\in (I'_n,J'_n)_{n \in
  \omega}$ since by the choice of $u_{n,m}$ (or property of $s_n$) $x \rest I'_m \not \in J'_m$ for all $m$.
\end{proof}

Suppose that $(I_n,J_n)_{n \in \omega}$ and $(I'_n,J'_n)_{n \in
  \omega}$ are two small sets and $(I_n,J_n)_{n \in \omega}
\subseteq (I'_n,J'_n)_{n \in  \omega}$.
Consider the partition consisting of sets
$\{I_n \cap I'_m: n,m\in \omega\}$. 
For each non-empty set $I_n \cap I'_m$ we define $J''_{n,m} \subseteq
2^{I_n \cap I'm}$ as follows: 

$s\in J''_{n,m}$ if there is $t \in J'_m$ such that $s \rest (I_n \cap
I'_m) = t \rest (I_n \cap
I'_m) $ and for all
 $u \in 2^{I'_m \setminus I_n} \ t \rest (I_n \cap I'_m)
  ^\frown u \in J'_m$.

Observe that the definition of $J''_{n,m}$ does not depend on
$J_n$.  

Note that
$$\sum_{m,n\in
  \omega, I_n\cap I'_m\neq \emptyset} \dfrac{|J''_{m,n}|}{2^{|I_n\cap
  I'_m|} }=\sum_{m \in \omega}\sum_{n\in \omega, I_n\cap I'_m\neq
\emptyset} \dfrac{|J''_{m,n}|}{2^{|I_n\cap
  I'_m|} } = $$
$$\sum_{m \in \omega}\sum_{n\in \omega, I_n\cap I'_m\neq
\emptyset} \dfrac{|J''_{n,m}|\cdot
    2^{|I'_m\setminus  I_n|}}{2^{|I''_k|} \cdot 2^{|I'_m\setminus
      I_n|}} \leq \sum_{m \in \omega} \dfrac{|J'_m|}{2^{|I'_m|}}<\infty.$$

To finish the proof observe that for $x \in 2^\omega$, whenever $x \rest (I_n\cap I'_m) \in
J''_{n,m} $ then $x \rest I'_m \in J'_m$. 
Similarly, if $x \rest I_n \in J_n$ then by Lemma \ref{key} there is
$m$ such that $x \rest (I_n \cap I'_m) \in J''_{m,n}$
It follows that
$(I_n,J_n)_{n \in \omega} \subseteq (I_{n,m},J''_{n,m})_{n,m\in \omega} \subseteq (I'_m,J'_m)_{m \in
  \omega}$.
\end{proof}

\section{Small sets versus measure zero sets}
In this section we will prove the main result.

\begin{theorem}\label{main2}
There exists a null set which is not small, that is $\T \subsetneq \N$.
\end{theorem}

\begin {proof}
We will use the following:
\begin{lemma}\label{prob}
For every $\varepsilon >0$ and sufficiently large $n \in \omega$ there
exists a set $A \subset 2^n$ such that $\dfrac{|A|}{2^n} <
\varepsilon$ and for every $u \subset n$ such that $\dfrac{n}{4}\leq
|u| \leq \dfrac{3n}{4}$, and $B_0 \subset 2^u$ and $B_1 \subset 2^{n
  \setminus u}$ such that $\dfrac{|B_0|}{2^{|u|}} \geq \dfrac{1}{2} $ and
$\dfrac{|B_1|}{2^{|n\setminus u|}} \geq \dfrac{1}{2} $ we have $(B_0\times B_1)
\cap A \neq \emptyset$.
\end{lemma}
\begin{proof}
The key case is when $\varepsilon$ is very small and sets $B_0, B_1$
have relative measure approximately $\dfrac{1}{2}$. In such case
$B_0\times B_2$ has relative measure $\dfrac{1}{4}$ yet it intersects
$A$.
Fix large $n \in \omega$ and choose $A \subset 2^n$ randomly. That is,
for each $s \in 2^n$, the probability $\Prob(s \in A)=\varepsilon$ and for $s, s' \in
2^n$, events $s \in A$ and $s' \in A$ are independent. It is well
known that for a large enough $n$ the set constructed this way will
have measure $\varepsilon$ (with negligible error).

Fix $n/4 \leq |u| \leq 3n/4$ and let 
$${\mathcal B}_u=\left\{(B_0,B_1): B_0 \subset 2^u, \ B_1 \subset
2^{n\setminus u} \text{ and }\dfrac{|B_0|}{2^{|u|}},
\dfrac{|B_1|}{2^{|n\setminus u|}}\geq \dfrac{1}{2}\right\}.$$
Note that $|{\mathcal B}_u| \leq 2^{{2^{|u|}+ 2^{|n\setminus u|}}} \leq
2^{2^{\frac{3n}{4}+1}}.$

For $(B_0,B_1) \in {\mathcal B}_u$, $\Prob((B_0\times B_1) \cap
A=\emptyset)=(1-\varepsilon)^{|B_0\times B_1|}\leq (1-\varepsilon)^{2^{n-2}}$.
Consequently, 
$$\Prob(\exists (B_0,B_1) \in {\mathcal B}_u \ (B_0\times
B_1) \cap A =\emptyset) \leq |{\mathcal B}_u|
(1-\varepsilon)^{2^{{n-2}}} \leq
2^{2^{\frac{3n}{4}+1}}(1-\varepsilon)^{2^{{n-2}}}.$$
Finally, since we have at most $2^n$ possible sets $u$,
\begin{multline*}\Prob(\exists u\ \exists  (B_0,B_1) \in {\mathcal B}_u \ (B_0\times
B_1) \cap A =\emptyset) \leq \\
2^n |{\mathcal B}_u|
(1-\varepsilon)^{2^{{n-2}}} \leq
2^{2^{\frac{3n}{4}}+n+1}(1-\varepsilon)^{2^{{n-2}}} \leq
  2^{2^{\frac{7n}{8}}}(1-\varepsilon)^{2^{{n-2}}} \leq \\
2^{2^{\frac{7n}{8}}} (1-\varepsilon)^{\frac{1}{\varepsilon}\epsilon{2^{n-2}} }\leq
\frac{2^{2^{\frac{7n}{8}}} }{2^{\varepsilon{2^{n-2}}}} \longrightarrow 0 \text{ as } n \rightarrow \infty .
\end{multline*}
Therefore there is a non-zero  probability that a randomly chosen set
$A$ has the required properties. In particular, such a set must exist.

\end{proof}

Let $\{k^0_n,k^1_n: n \in \omega\} $ be two sequences defined as 
$k^0_{n}=n(n+1)$ and 
$k^1_{n}=n^2$ for $n >0$.

Let $I^0_n=[k^0_n,k^0_{n+1})$ and $I^1_n=[k^1_n,k^1_{n+1})$ for $n \in
\omega$.
Observe that the sequences are selected such that 
\begin{enumerate}
\item $|I^0_n|=2n+2$ and $|I^1_n|=2n+1$ for $n \in \omega$,
\item $I^0_n \subset I^1_{n} \cup I^1_{n+1}$ for $n > 0$,
\item $I^1_n \subset I^0_{n-1} \cup I^0_{n}$ for $n>1$,
\item $|I^0_n \cap I^1_{n}|=|I^1_n \cap I^0_{n-1}|=n$ for $n>1$,
\item $|I^0_n \cap I^1_{n+1}|=|I^1_n \cap I^0_{n}|=n+1$
  for $n >1$.
\end{enumerate}

Finally, for $n >0$ let $J^0_n \subset 2^{I^0_n}$ and $J^1_n
\subset 2^{I^1_n}$ be selected as in Lemma \ref{prob} for
$\varepsilon_n=\frac{1}{n^2}$.  Easy calculation shows  that for $n \geq 140$ the sets
$J^0_n$ and $J^1_n$ are defined and have the required properties.

Suppose that $(I^0_n,J^0_n)_{n\in \omega} \cup  (I^1_n,J^1_n)_{n\in
  \omega}\subset (I^2_n,J^2_n)_{n\in \omega} $.

{\sc Case 1}
There exists $i \in \{0,1\}$ and infinitely many $n,m \in \omega$ such
that 
$$\frac{|I^i_m|}{4} \leq |I^i_m \cap I^2_n| \leq
\frac{3|I^i_m|}{4}.$$

Without loss of generality $i=0$. Let $\{a_k: k\in
\omega\}$ be a partition of $\omega$ into finite sets. For $n \in
\omega$ define
$I'_n=\bigcup_{l\in a_n} I^2_l$ and $J'_n=\{s \in 2^{I'_n}: \exists l
\in a_n \ \exists t \in J^2_l \ s \rest I^2_l = t\rest I^2_l\}$. By Lemma
\ref{trivial}, we know that $(I'_n, J'_n)_{n \in
  \omega}=(I^2_n,J^2_n)_{n\in \omega}$ no matter what is the choice of
the partition $\{a_k: k\in
\omega\}$. 

Consequently, let us choose $\{a_k: k\in
\omega\}$ and an infinite set $Z \subseteq \omega$ such that 
\begin{enumerate}
\item for every
$m \in Z$ there is $n \in \omega$ such that $\dfrac{|I^0_m|}{4} \leq |I^0_m \cap I'_n| \leq
\dfrac{3|I^0_m|}{4}.$
\item for every $m \in Z$  there exists $n \in \omega$ such that
$I^0_m \subset I'_n \cup I'_{n+1}$,
\item for every $n \in \omega$ there is at most one $m \in Z$ such
  that $I^0_m \cap I'_n \neq \emptyset$.
\end{enumerate}

To construct the required partition $\{a_k: k\in \omega\}$ we inductively glue together sets $I^2_l$ as
follows: suppose that $m$ is such that there is $n$ such
that 
$\dfrac{|I^0_m|}{4} \leq |I^0_m \cap I^2_n|\leq 
\dfrac{3|I^0_m|}{4}.$ Then we define $a_n=\{n\}$ and $a_{n+1} =\{u:
I^0_m \cap I^2_u \neq \emptyset  \text{ and }u\neq
n\}$. Let $Z$ be the subset of the collection of $m$'s selected as
above that is thin enough to satisfy condition (3).

Recall that  $(I^0_n,J^0_n)_{n \in \omega} \subseteq (I^2_n,J^2_n) _{n \in
  \omega}=(I'_n,J'_n) _{n \in \omega}$.

Working towards contradiction fix $m \in Z$, and let $I^0_m \subseteq
I'_n \cup I'_{n+1}$ (in this case $I'_n=I^2_n$). By Lemma 
\ref{key} it follows that if $m$ is large enough then for every $s \in J^0_m$ either 
\begin{enumerate}
\item for every $u \in 2^{I'_n \setminus I^0_m}$ we have $ s \rest (I^0_m \cap
  I'_n )\ ^\frown u \in J'_n$, or 
\item for every $u \in 2^{I'_{n+1} \setminus I^0_m}$ we have $ s \rest (I^0_m \cap
  I'_{n+1} )\ ^\frown u \in J'_{n+1}$.
\end{enumerate}

Let $J''_n= \{s \in 2^{ I^0_m \cap I'_n}:  \forall u \in 2^{I'_n
  \setminus I^0_m} \ s  ^\frown u \in J'_n\}$ and 
 $J''_{n+1}= \{s \in 2^{I^0_m \cap I'_{n+1}}: \forall u \in
 2^{I'_{n+1} \setminus I^0_m} \ s   ^\frown u \in J'_{n+1}\}$.

Clearly $\dfrac{|J''_n|}{2^{|I'_n\cap I^0_m|}}\leq
\dfrac{|J'_n|}{2^{|I'_n|}} \leq \dfrac{1}{2}$ and $\dfrac{|J''_{n+1}|}{2^{|I'_{n+1}\cap I^0_m|}}\leq
\dfrac{|J'_{n+1}|}{2^{|I'_{n+1}|}} \leq \dfrac{1}{2}$.

Let
$B_n= 2^{I^0_m \cap I'_n} \setminus J'_n$ and 
$B_{n+1}=2^{I^0_m \cap I'_{n+1}} \setminus J'_{n+1}$.

It follows that 
$\dfrac{|B_n|}{2^{|I^0_m \cap I'_n|}},  \dfrac{|B_{n+1}|}{2^{|I^0_m \cap
    I'_{n+1}|}} \geq \dfrac{1}{2}$. 
By Lemma \ref{prob} and the definition of set $(I^0_m,J^0_m)_{m \in
  \omega}$ there is $s_m
\in (B_n \times B_{n+1}) \cap J^0_m$. Consequently  
 there is $t_m \in 2^{I'_n \cup I'_{n+1}}$ such that $t_m \rest I^0_m =s_m
 \in J^m_0$ but 
$t_m \rest I'_n \not \in J'_n$ and $t_m 
\rest I'_{n+1} \not \in J'_{n+1}$. 
For each $n \in \omega$ choose $r_n \in 2^{I'_n} \setminus
J'_n$. Define $x \in 2^\omega$ as
$$x \rest I'_n=\left\{\begin{array}{ll}
t_m \rest I'_n & \text{if } I^0_m \cap I'_n \neq \emptyset\\
r_n & \text{if } I^0_m\cap I'_n= \emptyset \text{ for all } m \in
      Z
\end{array} \right. . $$
It follows that $x \in (I^0_n,J^0_n)_{n \in \omega}$ but $x \not \in
(I'_n,J'_n) _{n \in \omega} = (I^2_n,J^2_n) _{n \in
  \omega}$, contradiction.

\bigskip 

{\sc Case 2} For every $i \in \{0,1\}$, almost every $n \in
\omega$ and every $m \in \omega$,
$$|I^2_n \cap I^i_m| \leq \dfrac{|I^i_m|}{4}.$$

This is quite similar to the previous case. 

We inductively  choose $\{a_k: k\in
\omega\}$ and define $I'_n$'s and $J'_n$'s as before. 
Next construct an infinite set $Z \subseteq \omega$ such that 
\begin{enumerate}
\item for every $m \in Z$  there exists $n \in \omega$ such that
$I^0_m \subset I'_n \cup I'_{n+1}$ and $\dfrac{|I^0_m|}{4} \leq |I^0_m
\cap I'_n|, \ |I^0_m \cap I'_{n+1}|\leq
\dfrac{3|I^0_m|}{4}.$
\item for every $n \in \omega$ there is at most one $m \in Z$ such
  that $I^0_m \cap I'_n \neq \emptyset$.
\end{enumerate}
Since $|I^2_k \cap I^i_m| \leq \dfrac{|I^i_m|}{4}$ for each $k,m$ we
can get (1) by careful splitting $\{k: I^0_m \cap I^2_k \neq
  \emptyset\}$ into two sets.

The rest of the proof is exactly as before.

\bigskip 

To conclude the proof it suffices to show that these two cases exhaust
all possibilities. To this end we check  that if for some
$i\in \{0,1\}$, $m,n\in \omega$, $|I^2_n  \cap I^i_m| >
\dfrac{3|I^1_m|}{4}$ then for some $j \in \{0,1\}$ and $k \in \omega$,
$$\dfrac{3|I^j_k|}{4} \leq |I^2_n  \cap I^j_k| \leq  
\dfrac{3|I^j_k|}{4}.$$
This will show that potential remaining cases are already included in  the {\sc Case 1}.

Fix $i=0$ and $n \in \omega$ (the case $i=1$ is analogous.)

By the choice of intervals $I^0_m$ and $I^1_m$, it
follows that if $|I^2_n  \cap I^0_m| >
\dfrac{3|I^0_m|}{4}$  then $|I^2_n  \cap I^1_m| >
\dfrac{|I^1_m|}{4}$. If $|I^2_n  \cap I^1_m| \leq
\dfrac{3|I^1_m|}{4}$ then we are in {\sc Case 1}. Otherwise $|I^2_n  \cap I^1_m| >
\dfrac{3|I^1_m|}{4}$ and so $|I^2_n  \cap I^0_{m+1}| >
\dfrac{|I^1_{m+1}|}{4}$. Continue inductively until the construction
terminates after finitely many steps settling on $j$ and $k$.

\end{proof}

\begin{theorem}
Not every small set is small$^\star$, that is $\T^\star \subsetneq \T$.
\end{theorem}

\begin{proof}
The proof is a modification of the previous argument.

Let $I^0_n,I^1_n, J^0_n$ and $J^1_n$ for $n \in \omega$ be like in
the proof of \ref{main2}. 
Let $\bar{I}^0_n=\{2k: k \in I^0_n\}$ and $\bar{I}^1_n=\{2k+1: k\in
I^1_n\}$ for $n \in \omega$ and let $\bar{J}^0_n \subset 2^{\bar{I}^0_n},
\bar{J}^1_n\subset 2^{\bar{I}^1_n}$ for $n \in
\omega$ be the induced sets.
Note that $(\{\bar{I}^0_n, \bar{I}^1_n\}, \{\bar{J}^0_n,\bar{J}^1_n\})_{n \in
  \omega}$ is a small set. We will show that this set is not
small$^\star$.
Suppose that $(\{\bar{I}^0_n, \bar{I}^1_n\}, \{\bar{J}^0_n,\bar{J}^1_n\})_{n \in
  \omega} \subseteq (I_n,J_n)_{n \in \omega}$, where
$I_n=[k_n,k_{n+1}) $ for an increasing sequence $\{k_n: n\in
\omega\}$. 

Without loss of generality we can assume that for every $n \in \omega $ there
exists $i \in \{0,1\}$ and $m \in \omega$ such that

\begin{enumerate}
\item $I^i_m \subseteq I_{n} \cup I_{n+1}$,
\item $\dfrac{|I^i_m|}{4} \leq |I_n\cap I^i_m| \leq
  \dfrac{3|I^i_m|}{4}$,
\item $\dfrac{|I^i_m|}{4} \leq |I_{n+1}\cap I^i_m| \leq
  \dfrac{3|I^i_m|}{4}$.
\end{enumerate}

To get (1) we combine consecutive intervals $I_n$ to make sure that
each $I^i_m$ belongs to at most two of them. Points (2) and (3) are a
consequence of the properties of the original sequences $\{I^0_n,I^1_n:
n \in \omega\}$, namely that each integer  belongs to exactly two of these
intervals and that intersecting intervals cut each other approximately
in half. The following example illustrates the procedure for finding
$i$ and $m$: If $k_n$ is
even then $k_n/2 $ belongs to $I^0_j \cap I^1_k$ with $k-j$ equal to 0
or 1. The value of $i$ and $m$ depend on whether $k_n/2$
belongs to the lower or upper half of the said interval. The case when
$k_n $ is odd is similar.

The rest of the proof is exactly like Case  1 of Theorem \ref{main2}.
\end{proof}

\end{document}